\documentclass[11pt]{article}
\usepackage{amssymb,amsmath,amsfonts, amscd}
\usepackage{array}
\usepackage{amsmath}
\usepackage{amssymb}

\newtheorem{theorem}{Theorem}

\newtheorem{lemma}[theorem]{Lemma}

\newtheorem{definition}[theorem]{Definition}

\newcommand{\qedbox}{$\blacksquare$ \newline}
\newenvironment{proof}%
{%
\noindent{\it Proof.} } { \hfill\qedbox }

\newcommand{\proofend}{\hfill\qedbox}

\newcommand{\itsp}{\vspace{-0pt}\item}

\input epsf

\title{On Avoider-Enforcer games}
\author{J\'ozsef Balogh\thanks{Department of Mathematics, University
of Illinois at Urbana-Champaign, Urbana, IL, 61801,
\texttt{jobal@math.uiuc.edu}. This author's research supported in
part by NSF CAREER Grant DMS-0745185 and DMS-0600303, UIUC Campus Research Board Grants 06139, 07048 and 08086, and OTKA grant 049398.} \and Ryan Martin\thanks{Department of Mathematics, Iowa
State University, Ames, IA 50011, \texttt{rymartin@iastate.edu}.
This author's research supported in part by NSA grant
H98230-05-1-0257.}}

\date{\today}

\newcommand{\PP}{\mathcal{P}}
\newcommand{\C}{\mathcal{C}}
\newcommand{\NC}{\mathcal{NC}}
\newcommand{\F}{\mathcal{F}}
\newcommand{\E}{\mathcal{E}}
\newcommand{\A}{\mathcal{A}}
\newcommand{\B}{\mathcal{B}}

\newcommand{\pnf}{{\mathcal{P}_n(\mathcal{F})}}
\newcommand{\pnfind}{{\mathcal{P}_n^{\rm ind}(\mathcal{F})}}

\begin{document}
\maketitle

\begin{abstract}
In the Avoider-Enforcer game on the complete graph $K_n$, the
players (Avoider and Enforcer) each take an edge in turn.  Given a graph property $\mathcal{P}$, Enforcer wins the game if Avoider's graph has the property $\mathcal{P}$.  An important parameter is $\tau_E({\cal P})$, the smallest integer $t$ such that Enforcer can win the game against any opponent in $t$ rounds.

In this paper, let $\mathcal{F}$ be an arbitrary family of graphs and $\mathcal{P}$ be the property that a member of $\mathcal{F}$ is a subgraph or is an induced subgraph.  We determine the
asymptotic value of $\tau_E(\mathcal{P})$ when $\mathcal{F}$
contains no bipartite graph and establish that
$\tau_E(\mathcal{P})=o(n^2)$ if $\mathcal{F}$ contains a bipartite
graph.

The proof uses the game of JumbleG and the Szemer\'edi Regularity Lemma.
\end{abstract}

\section{Introduction}

An \textbf{unbiased positional game} is one in which two players
alternately select a vertex from a hypergraph. In the most common
formulation of the game, one player is Maker and the other Breaker.
Maker attempts to occupy all vertices in some hyperedge and Breaker
attempts to occupy at least one vertex in every hyperedge.  In the
graph context, players select edges of a complete graph $K_n$ or possibly some
other graph, such as a random graph \cite{Gnp}. Here Maker attempts
to create a graph with a given monotone property and Breaker
attempts to prevent Maker from achieving this.

Hefetz, Krivelevich and Szab\'o \cite{AvEn} recently investigated
the so-called \textbf{Avoider-Enforcer game}.  The players are,
unsurprisingly, Avoider and Enforcer.  In the context of graph
games, Avoider attempts to prevent having her edges induce a graph
with a property $\PP$ for as many rounds as possible. Enforcer
selects his edges in such a way as to force \textbf{Avoider's graph}
to have property $\PP$ as early as possible.

At the end of $r$ rounds, Avoider and Enforcer each have chosen
exactly $r$ edges.  Let $\tau_E(\PP)$ be the smallest integer $t$
such that Enforcer can win the $\PP$-property game in $t$ rounds.
Let $\C_n^k$ denote the property that an $n$-vertex graph is
$k$-colorable and let $\NC_n^k$ denote the property that an
$n$-vertex graph is not $k$-colorable.

In \cite{Conj2} (see also \cite{Conj1} for similar questions on Maker-Breaker games), Hefetz, Krivelevich, Stojakovi\'c and Szab\'o
establish that
$$
\frac{n^2}{8}+\frac{n-2}{12}\leq\tau_E(\NC_n^2)\leq\frac{n^2}{8}+\frac{n}{2}+1,
$$
but for $k\geq 3$  show only that
$$
(1-o(1))\frac{(k-1)n^2}{4k}\leq\tau_E(\NC_n^k)<\frac{1}{2}\binom{n}{2}
. $$ They state that it ``seems reasonable'' that
$\tau_E(\NC_n^k)\leq (1+o(1))\frac{(k-1)n^2}{4k}$.

In this paper, we prove that this is the case, as a consequence of a
stronger result.

\subsection{Main Results}
Let $t(n,k)$ denote the Tur\'an number, which is the maximum number of
edges in a graph on $n$ vertices with no $K_{k+1}$.  In particular, Tur\'an's theorem gives
$\left(\frac{k-1}{k}\right)\frac{n^2}{2}-\frac{k}{8}\leq t(n,k)\leq
\left(\frac{k-1}{k}\right)\frac{n^2}{2}$.

\begin{definition}
  For a family  of graphs, $\mathcal{F}$, denote
   $\pnf$ to be the property that a graph on
  $n$ vertices has a copy of some member of
  $\F$ as a subgraph and let $\pnfind$ be the property that a graph
  on $n$ vertices has a copy of some member of $\F$ as an
  \textbf{induced} subgraph.
\end{definition}

\begin{theorem}
  Let $\F$ be a family of graphs such that $k=\min\{\chi(F) : F\in\F\}\geq 3$.
    In the Avoider-Enforcer game for properties $\pnf$ and $\pnfind$, for $n$  large enough, we have
  $$ \left(\frac{k-2}{k-1}\right)\frac{n^2}{4}-O(k)=
     \left\lfloor\frac{1}{2} t(n,k-1)\right\rfloor\leq
     \tau_E(\pnf)\leq\tau_E(\pnfind)\leq
     \left(\frac{k-2}{k-1}\right)
     \frac{n^2}{4}+o(n^2) . $$
  In addition, the last two inequalities hold if $k=2$.
   \label{mainth}
\end{theorem}

The last inequality in Theorem \ref{mainth} is our main result.
Enforcer's strategy is simple, he tries to achieve to maintain that
the board ``looks like a random graph'' after each round. We simply
show that in this case Avoider's graph when it has
$\frac{k-2}{k-1}\left(\frac{n^2}{4}\right)+o(n^2)$ edges, will
satisfy property $\pnfind$. The answer to the question of Hefetz, Krivelevich, Stojakovi\'c and Szab\'o is a corollary of Theorem~\ref{mainth}:

\begin{theorem}
  Let $k\geq 2$ be an integer and let $\NC_n^k$ be the property that
  a graph is not $k$-colorable.  In the Avoider-Enforcer game,
  $$ \left(\frac{k-1}{k}\right)\frac{n^2}{4}-O(k)=
  \left\lfloor\frac{1}{2} t(n,k)\right\rfloor\leq
     \tau_E(\NC_n^k)\leq
     \left(\frac{k-1}{k}\right)\frac{n^2}{4}+o(n^2) . $$
  \label{color}
  In addition, the last two inequalities hold if $k=1$.
\end{theorem}

Let $\F=\{K_{k+1}\}$. The upper bound for Theorem \ref{color} is an
immediate consequence of Theorem~\ref{mainth}.  The lower bound
comes from the same trivial strategy that establishes the lower
bound in Theorem~\ref{mainth}.

This result is weaker than that in \cite{Conj2} for $k=2$, but gives
the correct asymptotic approximation for all $k\geq 2$.

This paper is organized as follows: In Section \ref{sec:jumbleg}, we
describe the game of JumbleG \cite{JumbleG}, along with important
consequences thereof. In Section \ref{sec:reglem}, we state the
version of the Szemer\'edi Regularity Lemma that is useful for our
purposes. Section \ref{sec:proofs} contains the proofs of the
theorems. Section \ref{sec:conc} contains some concluding remarks.

\section{The game of JumbleG}
\label{sec:jumbleg}

The game of JumbleG is a traditional Maker-Breaker game.  The goal
of Maker is to create a pseudo-random graph.  Frieze, Krivelevich,
Pikhurko and Szab\'o \cite{JumbleG} use two different conditions for
pseudo-randomness.  We will use their first definition, including the use of ``$\epsilon$-regular'' to describe a graph, which is not to be confused with $\epsilon$-regular pairs, as defined by Szemer\'edi's Regularity Lemma.

\begin{definition}  If $G$ is a graph and $S,T$ are disjoint vertex
sets in $V(G)$, then denote by $e_G(S,T)$ to be the number of edges
which have one endpoint in $S$ and the other in $T$.

   A   pair of disjoint vertex sets $(S,T)$ is
   \textbf{$\epsilon$-unbiased} if
   $$ \left|\frac{e_G(S,T)}{|S||T|}-\frac{1}{2}\right|\leq\epsilon. $$

  A graph $G$ on $n$ vertices, with minimum degree $\delta(G)$,
  is \textbf{$\epsilon$-regular} if both:
   \begin{itemize}
      \item[\textbf{P1}] $\delta(G)\geq (1/2-\epsilon)n$.
      \item[\textbf{P2}] Any pair $S,T$ of disjoint subsets of $V(G)$ with
      $|S|,|T|>\epsilon n$ is $\epsilon$-unbiased.
   \end{itemize}
\end{definition}

Formally, the game of JumbleG$(\epsilon)$ is won by Maker if Maker can ensure that his graph is an $\epsilon$-regular graph.

\begin{theorem}[\cite{JumbleG}]\label{thm:JumbleG}
   Maker has a winning strategy in JumbleG$(\epsilon)$, provided
   $\epsilon\geq 2(\log n/n)^{1/3}$ and $n$ is sufficiently large.
\end{theorem}

In our proofs, Enforcer will simply play as Maker in the game of
JumbleG$(\epsilon)$. 
Note that the game stops before all the edges are claimed. So Theorem~\ref{thm:JumbleG} is used so that if the game were continued further, then from the resulting graph, Enforcer could win the game. As a result, between any pair of disjoint sets, both of them large enough, Avoider cannot occupy too many of the edges.  We formalize this as follows.

\begin{definition} After round $r$ in the game, let
$e_B(S,T;r)$ denote the number
   of edges that Breaker (Avoider) occupies in the pair $(S,T)$ and $e_M(S,T;r)$
   denote the number of edges that Maker (Enforcer) occupies.
\end{definition}
\begin{lemma}
   Let $S,T\subset V(G)$ be disjoint sets and $|S|,|T|\geq\epsilon
   n$.   If Maker plays
   a strategy to win JumbleG$(\epsilon)$, then for every  $r$
   $$ e_B(S,T;r)-e_M(S,T;r)\leq2\epsilon|S||T| +1 . $$
   \label{jumble}
\end{lemma}

\begin{proof}
Fix any integer $r>0$.   Suppose
$e_B(S,T;r)-e_M(S,T;r)>2\epsilon|S||T|+1$. From this point on, we
let Breaker to occupy edges only between $S$ and $T$ whenever
possible. When the game concludes, the number of edges that
   Breaker occupies between $S$ and $T$ is at least
   \begin{eqnarray*}
      \lefteqn{e_B(S,T;r)
               +\left\lfloor\frac{1}{2}\left(|S||T|-e_B(S,T;r)-e_M(S,T;r)
                                       \right)\right\rfloor} \\
      & \geq & e_B(S,T;r)
               +\frac{1}{2}\left(|S||T|-e_B(S,T;r)-e_M(S,T;r)
                           \right)-\frac{1}{2}\\
      & > & \frac{1}{2}|S||T|
               +\frac{1}{2}\left(2\epsilon|S||T|+1\right)-\frac{1}{2}
               \\
      & \geq & \left(\frac{1}{2}+\epsilon\right)|S||T| .
   \end{eqnarray*}
   This contradicts the assumption that Maker played a winning
   strategy for JumbleG$(\epsilon)$.
\end{proof}

\section{The Regularity Lemma}
\label{sec:reglem}

\begin{definition}
   Let $A$ and $B$ be disjoint vertex sets.  The number of edges
   with one endpoint in $A$ and the other in $B$ is denoted by
   $e(A,B)$.
   The \textbf{density} of $(A,B)$ is denoted by
   $$ d(A,B)=\frac{e(A,B)}{|A||B|} . $$
   A pair $(A,B)$ is \textbf{$\alpha$-regular} if, for every
   $X\subseteq A$ and $Y\subseteq B$ with $|X|>\alpha|A|$ and
   $|Y|>\alpha|B|$,
   $$ |d(A,B)-d(X,Y)|<\alpha . $$
A partition $V_1,\ldots,V_\ell$ is an \textbf{equipartition} of
$V$ if for every $i,j$ we have $\left||V_i|-|V_j|\right|\le 1$.
\end{definition}

We use a form of the Szemer\'edi Regularity Lemma, introduced by Alon,
Fischer, Krivelevich and M. Szegedy \cite{reglemplus}.
Lemma~\ref{reglem} is a simplified version.

\begin{lemma}[Regularity Lemma~\cite{reglemplus}]\label{reglem}
   For every integer $m$ and constant
   $\E>0$, there is an $S=S(m,\E)$ which satisfies the
   following:  For any graph $G$ on $n\geq S$ vertices, there exists
   an equipartition $\A=\{V_i : 1\leq i\leq\ell\}$ of $V(G)$ and an
   induced subgraph $U$ of $G$, with an equipartition $\B=\{U_i :
   1\leq i\leq\ell\}$ of the vertices of $U$, that satisfy:
   \begin{itemize}
      \itsp $m\leq\ell\leq S$.
      \itsp $U_i\subseteq V_i$ and $|U_i|=L\geq
      \lceil n/S\rceil$, for all $i\geq 1$.
      \itsp In the equipartition $\B$, \textbf{all} pairs are
      $\E$-regular.
      \itsp All but at most $\E\binom{\ell}{2}$ of the pairs $1\leq
      i<j\leq\ell$ are such that $|d(V_i,V_j)-d(U_i,U_j)|<\E$.
   \end{itemize}
\end{lemma}

Since we want to establish that a large enough density in $G$ will
enable us to apply Tur\'an's Theorem, we need to bound $e(U)$ in
terms of $e(G)$.
\begin{lemma}
   Let  $m$ be  an integer, and $\E$ be a constant such that
    $m\geq\E^{-1}$ and let $S=S(m,\E)$ be the integer provided
    by Lemma \ref{reglem}.  Let $n\geq 2S/\E$ and $G$ be a graph on
    $n$ vertices.  Let $U$ an the induced subgraph
      and    an equipartition of $V(U)$ be
      $\left\{U_i\right\}_{i=1}^{\ell}$
       with $|U_1|=\ldots=|U_{\ell}|=L$, provided by Lemma \ref{reglem}. Let
        $\tilde{U}$ be the graph formed by deleting all edges
        which have both endpoints in the same set $U_i$ for $i=1,\ldots,\ell$.  In this case,
   $$ e(U)\geq e(\tilde{U})\geq e(G)\frac{\ell^2L^2}{n^2}-3\E\ell^2L^2 . $$
   \label{density}
\end{lemma}

\begin{proof}
First, we bound $e(G)$, using the conditions in Lemma \ref{reglem}.
\begin{eqnarray}
   e(G) & \leq &
   \sum_i\binom{|V_i|}{2}
   +\sum_{1\leq i<j\leq\ell}e(V_i,V_j) \nonumber \\
   & \leq & \ell\frac{\left\lceil n/\ell\right\rceil^2}{2}
   +\E\binom{\ell}{2}\left\lceil\frac{n}{\ell}\right\rceil^2
   +\left\lceil\frac{n}{\ell}\right\rceil^2
    \sum_{1\leq i<j\leq\ell}\left[d(U_i,U_j)+\E\right] \nonumber \\
   & \leq & \frac{\ell}{2}\left\lceil\frac{n}{\ell}\right\rceil^2
   +2\E\binom{\ell}{2}\left\lceil\frac{n}{\ell}\right\rceil^2
   +{\left\lceil\frac{n}{\ell}\right\rceil^2}
    \sum_{1\leq i<j\leq\ell}d(U_i,U_j) \nonumber \\
   & \leq & \frac{\ell}{2}\left\lceil\frac{n}{\ell}\right\rceil^2
   +\E\ell^2\left\lceil\frac{n}{\ell}\right\rceil^2
   +\frac{\lceil n/\ell\rceil^2}{L^2}e(\tilde{U}). \label{eq:GU}
\end{eqnarray}

Recall that $L=|U_1|=\ldots=|U_{\ell}|$.  The calculations below,
use the fact that $\ell^{-1}\leq m^{-1}\leq\E$ and $\ell/n\leq
S/n\leq\E/2$.  By rearranging the terms in inequality (\ref{eq:GU}),
we get a lower bound for $e(\tilde{U})$.
\begin{eqnarray*}
   e(\tilde{U}) & \geq & e(G)\frac{L^2}{\left\lceil n/\ell\right\rceil^2}
   -\frac{\ell L^2}{2}-\E\ell^2L^2 \\
   & \geq & e(G)\frac{\ell^2L^2}{n^2}
   -e(G)\left(\frac{\ell^2L^2}{n^2}-\frac{L^2}{\lceil
   n/\ell\rceil^2}\right)-\frac{1}{2\ell}\ell^2L^2-\E\ell^2L^2 \\
   & \geq & e(G)\frac{\ell^2L^2}{n^2}
   -\frac{n^2}{2}\left(\frac{\ell^2L^2}{n^2}
                       -\frac{\ell^2L^2}{(n+\ell)^2}\right)
   -\frac{\E}{2}\ell^2L^2-\E\ell^2L^2 \\
   & = & e(G)\frac{\ell^2L^2}{n^2}
   -\left(1-\frac{n^2}{(n+\ell)^2}\right)\frac{\ell^2L^2}{2}
   -\frac{\E}{2}\ell^2L^2-\E\ell^2L^2 \\
   & \geq & e(G)\frac{\ell^2L^2}{n^2}-3\E\ell^2L^2.
\end{eqnarray*}

Trivially, $e(U)\geq e(\tilde{U})$ and this concludes the proof.
\end{proof}

Lemma \ref{regclique} establishes that a regular $f$-tuple will
induce any graph on $f$ vertices, given necessary density conditions.  We take the version from Alon and Shapira \cite{AlonShapira}. (It has appeared previously in other forms, see \cite{oldsurvey} and \cite{BolThom}.)
\begin{lemma}
   For every real $\eta$, $0<\eta<1$, and integer $f\geq 1$, there
   exists a $\gamma=\gamma(\eta,f)$ with the following property:
   Suppose that $H$ is a graph on $f$ vertices $v_1,\ldots,v_f$, and
   that $U_1,\ldots,U_f$ is an $f$-tuple of disjoint nonempty vertex
   sets of a graph $G$ such that for every $1\leq i<j\leq f$, the
   pair $(U_i,U_j)$ is $\gamma$-regular.  Moreover,
   whenever $(v_i,v_j)\in E(H)$ we have $d(U_i,U_j)\geq\eta$ and whenever $(v_i,v_j)\not\in E(H)$ we have $d(U_i,U_j)\leq 1-\eta$.
    Then, some $f$-tuple $u_1\in U_1,\ldots,u_f\in U_f$ spans an
   \textbf{induced} copy of $H$, where each $u_i$ plays the role of
   $v_i$.
\label{regclique}
\end{lemma}

The Slicing Lemma (Fact 1.5 in~\cite{oldsurvey}) is a basic fact of
regular pairs, common in proofs involving the Regularity Lemma.
\begin{lemma}[Slicing Lemma]
   Let $(U_i,U_j)$ be an $\alpha$-regular pair with density $d$ and
   $|U_i|=|U_j|=L_0$.  If $X\subseteq U_i$ and $Y\subseteq U_j$ with
   $|X|\geq L_i$ and $|Y|\geq L_j$, then $(X,Y)$ is $\alpha'$-regular, where
   $$ \alpha'=\max\left\{2\alpha, \frac{L_0}{L_i}\,\alpha,
   \frac{L_0}{L_j}\,\alpha\right\} , $$
   with density in $(d-\alpha,d+\alpha)$.
   \label{slicing}
\end{lemma}

\section{Proof of Theorem \ref{mainth}}
\label{sec:proofs}

As in many proofs involving the Regularity Lemma, there is a
sequence of constants.  Fix an $F\in\mathcal{F}$ with chromatic number
$k$ and order $f$.  With $a\ll b$ meaning that $a$ is small enough
relative to $b$, the constants are
$$ \epsilon\ll\E_0\ll\E_1\ll\eta\ll\delta\ll f^{-1} . $$
We will determine the precise relationships later.

Let $k\geq 3$.  In order to establish the lower bound for $\tau_E(\pnf)$, Avoider
equipartitions the $n$ vertices into $k-1$ clusters, and   chooses
edges only  between different clusters.  By this strategy, Avoider's
graph will always be $(k-1)$-colorable, thus, have no member of $\F$
as a subgraph.  Avoider can make at least $\lfloor
(1/2)t(n,k-1)\rfloor$ moves.  Careful calculations establish that
$$ t(n,k)=\left(\frac{k-1}{k}\right)\frac{n^2}{2}-\frac{k}{2}
   \left(\left\lceil\frac{n}{k}\right\rceil-\frac{n}{k}\right)
   \left(\frac{n}{k}-\left\lfloor\frac{n}{k}\right\rfloor\right) .
   $$

Having established the first inequality of Theorem \ref{mainth} for $k\geq 3$,
 we assume for the rest of the proof that $k\geq 2$. Since the existence of an induced
 copy of some $F\in\F$ implies the
existence of a copy of $F$ as a subgraph, it is trivial that
$\tau_E(\pnf)\leq\tau_E(\pnfind)$.

Finally, we shall prove the upper bound.  Among all $F\in\F$ choose
an  $F$ with $\chi(F)=k$ and let $f=|V(F)|$.  Enforcer will play the game JumbleG$(\epsilon)$,
 and we shall prove that Avoider's graph, after $\left(\frac{k-2}{4(k-1)}+o(1)\right)n^2$ rounds
 will contain $F$ as an induced subgraph.  Assume, for the sake of contradiction, that for (a small)
  $\delta>0$, Avoider managed to build a graph $G$ of order
$n$ (where $n$ is sufficiently large) and $\left(\frac{k-2}{4(k-1)}+2\delta\right)n^2$ edges.
  Apply the Regularity Lemma (Lemma
\ref{reglem}) to Avoider's graph with parameters $\E_0$ and
$m:=\max\{k,\lceil \E_0^{-1}\rceil \}$, (where $\E_0$ will be
determined later) which yields the subsets $U_1,\ldots,U_{\ell_0}$
and a constant $S_0=S_0(m,\E_0)$ so that
$|U_1|=\ldots=|U_{\ell_0}|=L_0\geq n/S_0$. Construct an auxiliary graph
$H$ on the vertex set $\{1,\ldots,\ell_0\}$ where $i\sim j$ if and
only if $d(U_i,U_j)\geq\delta+\epsilon$.  Let $\tilde U$ be the
graph induced by $U_1\cup\ldots\cup U_{\ell_0}$, with the edges
inside of each cluster deleted. We use $e(\tilde U)$ in order to
compute $e(H)$.  Observe that Lemma \ref{jumble} gives that
$d(U_i,U_j)\leq(1/2+\epsilon)$ as long as $L_0\geq n/S_0>\epsilon
n$.
\begin{eqnarray*}
   e(\tilde U) & \leq & e(H)\left(\frac{1}{2}+\epsilon\right)L_0^2
   +\left[\binom{\ell_0}{2}-e(H)\right](\delta+\epsilon)L_0^2 \\
   & \leq & e(H)\left(\frac{1}{2}-\delta\right)L_0^2
   +\frac{\delta+\epsilon}{2}\,\ell_0^2L_0^2.
\end{eqnarray*}
We use Lemma \ref{density} to bound $e(\tilde U)$ by $e(G)$.
\begin{eqnarray}
   e(H)  \geq
   \frac{e(\tilde U)-\frac{\delta+\epsilon}{2}\,\ell_0^2L_0^2}
        {\left(\frac{1}{2}-\delta\right)L_0^2}
    \geq  \frac{e(G)\,\frac{\ell_0^2}{n^2}
                  -3\E_0\ell_0^2
                  -\frac{\delta+\epsilon}{2}\,\ell_0^2}
                 {1/2-\delta} . \label{eq:eH}
\end{eqnarray}
If \begin{equation} \delta\geq 2\E_0+\frac{\epsilon}{3},
\qquad\text{which is equivalent to}\qquad
2\delta\geq\frac{\delta}{2}+3\E_0+\frac{\epsilon}{2} ,
\label{bd:Eep}
\end{equation}
then
$$ e(G)\geq\left(\frac{k-2}{4(k-1)}+2\delta\right)n^2>
   \left(\frac{k-2}{4(k-1)}+3\E_0
         +\frac{\delta+\epsilon}{2}\right)n^2 . $$
Plugging this into (\ref{eq:eH}), we obtain that
$e(H)>\left(\frac{k-2}{k-1}\right)\frac{\ell_0^2}{2}$. Therefore,
$H$ contains a  $K_k$  by Tur\'an's Theorem.  Without loss of
generality, this copy of $K_k$ is spanned by
$\left(U_1,\ldots,U_k\right)$. Each pair $(U_i,U_j)$ is
$\E_0$-regular with density in the interval
$(\delta+\epsilon,1/2+\epsilon)$. \\

\noindent\textbf{Note:} At this stage of the proof, we could use, say, the Blow-up
lemma (see \cite{oldsurvey}) to show that if $n$ is large enough, a
 not necessarily induced copy of  $F$ occurs as a subgraph in $\left(U_1,\ldots,U_k\right)$,
hence in $G$. This provides an upper bound for $\tau_E(\pnf)$.
 However, we want the stronger result that produces an upper bound for $\tau_E(\pnfind)$.
 To prove such a result, we will apply Szemer\'edi's Regularity Lemma inside the clusters that
 were formed by its first application.  An alternative way to finish the proof would have been to use the Erd\H os-Stone Theorem instead of Tur\'an's Theorem.  \\

So, we apply the Regularity Lemma (Lemma \ref{reglem}) to the portion of
Avoider's graph in $U_i$ with parameters $\E_1$ and $m=f$, for each
$i=1,\ldots,k$ (where $\E_1$ will be determined later), which yields
the constants $\ell_i$, the  subclasses
$U_{i,1},\ldots,U_{i,\ell_i}$ and a constant $S_1=S_1(m,\E_1)$ so that
$|U_{i,1}|=\ldots=|U_{i,\ell_i}|=L_i\geq L_0/S_1\geq n/(S_0S_1)$. If
\begin{equation} \epsilon\leq S_0^{-1}S_1^{-1} \label{bd:eps}
\end{equation} then each pair $(U_{i,x},U_{j,y})$ has density at
most $1/2+\epsilon$, by Lemma \ref{jumble}.

Each pair $(U_{i,x},U_{i,y})$ is $\E_1$-regular.  Now consider a
pair $(U_{i,x},U_{j,y})$ where $i\neq j$.  The pair $(U_i,U_j)$ is
$\E_0$-regular with density at least $\delta+\epsilon$. Using the
Slicing Lemma (Lemma \ref{slicing}), the pair $(U_{i,x},U_{j,y})$ is
$\E_0L_0\cdot\max\{L_i^{-1},L_j^{-1}\}$-regular.  Since $L_i,L_j\geq
L_0/S_1$, the pair $(U_{i,x},U_{j,y})$ is $\E_0 S_1$-regular with
density at most $1/2+\epsilon$ and at least $\delta/2$, as long as
\begin{equation} \E_0\leq\delta/2+\epsilon . \label{bd:E0}
\end{equation}

Finally, we apply Lemma~\ref{regclique} with $\eta=\delta/2$ to $F$
and the tuple $\{U_{i,x}\}_{1\le i\le k, 1\le x\le \ell_i}$. This
implies the existence of the constant
$\gamma=\gamma\left(f,\delta/2\right)$. As $F$ is $k$-colorable
graph of order $f$, and $f\le \ell_i$ for every $i$, all pair of
classes are $\E_0 S_1$-regular. In order to apply
Lemma~\ref{regclique}, it is sufficient to have
\begin{equation}
\E_0 S_1\leq\gamma\left(f,\delta/2\right) . \label{bd:E0gam}
\end{equation}
This way $F$ can be embedded into $G$, Avoider's graph. \\

\textbf{The order of choosing the constants.} First, a forbidden
graph $F$ with chromatic number $k$, and order $f$ is fixed, 
and an arbitrary $\delta>0$ is chosen, assuming that Avoider built a
graph with $\left(\frac{k-2}{4(k-1)}+2\delta\right)n^2$ edges, where
$n>n(f,\delta).$ Note that if $\delta$ is too large, i.e.,
$\left(\frac{k-2}{4(k-1)}+2\delta\right)n^2>\binom{n}{2}/2$, then
our theorem is true by default. Then from Lemma~\ref{regclique},
with $\eta=\delta/2$ we obtain the existence of a $\gamma$, which is
less than $m^{-1}$. The
 Regularity Lemma (Lemma~\ref{reglem}) is applied to the classes
 $U_i$ with $\E_1=\gamma$ and $m=f$. This also implies the existence
 of an $S_1=S_1(\E_1,m)$.
  We need that the pairs $(U_{i,x}, U_{j,y})$ are
$\gamma$-regular, hence $\E_0 S_1\leq\gamma$ needed, so a positive
constant $\E_0< \gamma/S_1$ is chosen. The
 Regularity Lemma (Lemma~\ref{reglem}) is applied to  $G$ with $\E_0$ and $m=f$.
  This also implies the existence
 of an $S_0=S_0(\E_0,m)$.
For JumbleG$(\epsilon)$, we use $\epsilon=S_0^{-1}S_1^{-1}$. In order to apply
all these Lemmas, we need to assume that $n$ is sufficiently large.
This specific choice of constants satisfies the inequalities
(\ref{bd:Eep}), (\ref{bd:eps}),
(\ref{bd:E0}) and (\ref{bd:E0gam}).\\

To summarize, we proved that  for every $\delta>0$, if the
Avoider-Enforcer induced-$\F$-avoidance game is played for
$\left(\frac{k-2}{4(k-1)}+2\delta\right)n^2$ rounds, then the
Avoider's graph will contain an induced $F$, and the Enforcer will
win. Thus, $\tau_E(\pnfind)\leq\left(\frac{k-2}{k-1}\right)\frac{n^2}{4}+o(n^2)$. \proofend

\section{Concluding remarks}
\label{sec:conc}

The game $\NC^1_n$ is trivial, but the $\pnf$ question is not  for
$k=2$.  That is, the case in which $\F$ contains a bipartite graph.
Theorem \ref{mainth} gives only that, if $k=2$ and $\{F\}$ contains
a bipartite graph, then
$$ \tau_E(\pnf)\leq\tau_E(\pnfind)\leq
o(n^2) , $$ which is not very helpful. And certainly in the
non-induced case, we have trivially
$$\tau_E(\pnf)\leq t(n,F),$$
where $t(n,F)$ denotes the maximum number of edges in a graph with
$n$ vertices with no $F$ as a subgraph.

The authors believe the approach needed to prove
Theorem~\ref{mainth} may solve some other problems in the area of
positional
games,
  as our proof is  the first
instance of Szemer\'edi's Regularity Lemma being used in the context of positional game theory.

Finally, we remark that one can prove the non-induced case of $k=3$ without the use of the
 Szemer\'edi Regularity Lemma. \\


{\bf Acknowledgements:} We would like to thank Mathematisches
Forschungsinstitut Oberwolfach for hosting a miniworkshop on
positional games,
where we started our project. We are indebted to  Danny Hefetz,
Michael Krivelevich, Milo\u{s} Stojakovi\'c, and Tibor Szab\'o\ for
their helpful remarks.  We also thank the anonymous referees for
their valuable
 comments on the manuscript.

\end{document}